\documentclass[a4paper,12pt]{article}
\usepackage{latexsym}	
\usepackage{amsmath}
\usepackage{amsthm}		

\usepackage{graphicx}
\usepackage{txfonts}
\usepackage{bm}
\usepackage{color}
\usepackage{epic,eepic}

\newcommand{\RED}[1]{{\color{red}#1}} 
 \renewcommand{\RED}[1]{{#1}} 
\newcommand{\BLU}[1]{{\color{blue}#1}} 
 \renewcommand{\BLU}[1]{{#1}} 

\usepackage{geometry}
\numberwithin{equation}{section}

\newtheorem{theorem}{Theorem}[section]
\newtheorem{proposition}[theorem]{Proposition}
\newtheorem{lemma}[theorem]{Lemma}

\newtheorem{remark}{Remark}[section]
\newtheorem{example}{Example}[section]
\newtheorem{claim}[theorem]{Claim}

\newcommand{\OMIT}[1]{{\bf [OMIT:} #1 \ {\bf --- end OMIT] }}  
   \renewcommand{\OMIT}[1]{}            

\newcommand{\RR}{{\mathbb{R}}}
\newcommand{\ZZ}{{\mathbb{Z}}}

\newcommand{\Rinf}{\RR \cup \{ +\infty \}}

\newcommand{\vecone}{{\bf 1}}
\newcommand{\veczero}{{\bf 0}}
\newcommand{\dom}{{\rm dom\,}}
\newcommand{\supp}{\mbox{\rm supp\,}}

\newcommand{\unitvec}[1]{\bm{1}\sp{#1}}
\newcommand{\argmax}{\arg \max}
\newcommand{\argmin}{\arg \min}

\newcommand{\decmin}{{\rm decmin}}

\newcommand{\MM}{{M$_{2}$}}

\newcommand{\finbox}{\hspace*{\fill}$\rule{0.2cm}{0.2cm}$}
\newcommand{\todaye}{\the\year/\the\month/\the\day}

\begin{document}

\title{
Two Proofs of a Structural Theorem of 
\\  Decreasing Minimization on Integrally Convex Sets 
}

\author{
Kazuo Murota%
\thanks{
The Institute of Statistical Mathematics,
Tokyo 190-8562, Japan; 
and
Faculty of Economics and Business Administration,
Tokyo Metropolitan University, 
Tokyo 192-0397, Japan,
murota@tmu.ac.jp}
\ and 
Akihisa Tamura%
\thanks{Department of Mathematics, Keio University, 
Yokohama 223-8522, Japan,
aki-tamura@math.keio.ac.jp}
}

\date{January 2025 / April 2025}

\maketitle

\begin{abstract}
This paper gives two different proofs to 
a structural theorem of decreasing minimization 
(lexicographic optimization)
on integrally convex sets. 
The theorem states that
the set of decreasingly minimal elements of an integrally convex set
can be represented as the intersection of a unit discrete cube and a face
of the convex hull of the given integrally convex set.
The first proof resorts to 
the Fenchel-type duality theorem in discrete convex analysis
and the second is more elementary using Farkas' lemma.
\end{abstract}







\section{Introduction}
\label{SCintro}

This paper is concerned with 
decreasing minimization 
(or lexicographic optimization)
on discrete convex sets.
For any vector 
$x\in \RR\sp{n}$, 
let $x{\downarrow}$ denote
the vector obtained from $x$ by rearranging its components 
in descending order, i.e., 
$x{\downarrow} = (x{\downarrow}_{1}, x{\downarrow}_{2}, \ldots , x{\downarrow}_{n})$
with
$x{\downarrow}_{1} \ge x{\downarrow}_{2} \ge \cdots \ge x{\downarrow}_{n}$,
where
$x{\downarrow}_{i}$ denotes the $i$th component of $x{\downarrow}$. 
For $x = (2,5,2,1,3)$,
for example, we have
$x{\downarrow}=(5,3,2,2,1)$. 
For any vectors $x$ and $y$ of the same dimension,
we compare $x{\downarrow}$
and $y{\downarrow}$ lexicographically to
define notations $x <_{\rm dec} y$ and $x\leq_{\rm dec} y$
as follows:
\begin{itemize}
\item
$x <_{\rm dec} y$: \ 
$x{\downarrow} \ne y{\downarrow}$, and
$x{\downarrow}_{i}<y{\downarrow}_{i}$
for the smallest $i$ with
$x{\downarrow}_{i} \ne y{\downarrow}_{i}$.

\item
$x\leq_{\rm dec} y$: \ 
$x{\downarrow} = y{\downarrow}$ or $x <_{\rm dec} y$.
\end{itemize}
\noindent
For
$x = (2,5,2,1,3)$ and $y =(1,5,2,4,1)$,
for example, we have
$x{\downarrow}=(5,3,2,2,1)$ and $y{\downarrow}=(5,4,2,1,1)$.
Since
$x{\downarrow}_{1} = y{\downarrow}_{1}$ and
$x{\downarrow}_{2} < y{\downarrow}_{2}$,
we have
$x <_{\rm dec} y$.
For a given set $S$ of vectors,
an element $x$ of $S$ is called 
{\em decreasingly minimal}
(or {\em dec-min})
in $S$ if it is minimal in $S$ with respect to 
$\leq_{\rm dec}$, that is, if 
$x \leq_{\rm dec} y$ for all $y\in S$.
In general, dec-min elements may not exist 
(e.g., $S = \ZZ\sp{n}$)
and are not uniquely determined
(e.g., $S = \{ (1,2), (2,1) \}$).
We denote 
the set of all dec-min elements of $S$ by $\decmin(S)$.
The problem of finding a dec-min element of a given set $S$
is called the
{\em decreasing minimization problem}
on $S$.

Decreasing minimization on a base polyhedron
(in continuous variables)
was investigated in depth by Fujishige \cite{Fuj80,Fuj05book}. 
For discrete variables,
Frank and Murota investigated, 
in a series of papers
\cite{FM19partII,FM22partA,FM22partB,FM22fairsbmflow,FM23decminRZ,FM23fairflow},
decreasing minimization for various types of discrete convex sets
such as M-convex sets, integer flows, and
integer submodular flows, 
where an M-convex set
\cite{Mdcasiam} 
is a synonym for the set of integer points
in an integral base polyhedron. 

The following theorem reveals  
a matroidal structure of the set of dec-min elements of an M-convex set.

\begin{theorem}[{\cite[Theorem~5.7]{FM22partA}}]     \label{THdecminMunit}
An M-convex set $S$ has a dec-min element.
The set of dec-min elements of $S$ can be represented as 
\[
\decmin(S) = \{ z + \unitvec{X} \mid X \in \mathcal{B} \}
\]
with an integer vector $z$ and a matroid basis family $\mathcal{B}$.
In particular, $\decmin(S)$ is an M-convex set. 
\end{theorem}

For the intersection of two M-convex sets,
termed an \MM-convex set \cite{Mdcasiam},
the following theorem is known.

\begin{theorem}[{\cite[Corollary~1.2]{FM22fairsbmflow}, \cite[Theorem~13.25]{Mdcamarz24}}] 
    \label{THdecminM2unit}
An \MM-convex set $S$ has a dec-min element.
The set of dec-min elements of $S$ can be represented as
\[
\decmin(S) = \{ z + \unitvec{X} \mid X \in \mathcal{B}_{1} \cap \mathcal{B}_{2} \}
\]
with an integer vector $z$ and matroid basis families
$\mathcal{B}_{1}$ and $\mathcal{B}_{2}$.
In particular, $\decmin(S)$ is an \MM-convex set. 
\end{theorem}

In this paper we are concerned with the following similar statement
for an integrally convex set.
While postponing the precise definition of an integrally convex set 
to Section~\ref{SCintcnv}, we mention here
that integral convexity includes M-convexity and \MM-convexity as its special cases.

\begin{theorem}[{\cite[Theorem~13.23]{Mdcamarz24}}]     \label{THdecminICunit}
Let $S$ be an integrally convex set admitting a dec-min element.
The set of dec-min elements of $S$
can be represented as 
$\decmin(S)  = F \cap [z, z']_{\ZZ}$
with some face $F$ of 
the convex hull $\overline{S}$
and integer vectors $z$ and $z'$
satisfying $\veczero \le z' - z \le \vecone$.
In particular, $\decmin(S)$ is an integrally convex set. 
\end{theorem}

The assumption of integral convexity in Theorem~\ref{THdecminICunit}
cannot be removed, as follows.

\begin{example} \rm \label{EXdecminNICunit}
Let $S= \{ x, y \}$
with $x = (2,1,0,0)$ and $y = (0,0,1,2)$,
where $S$ is not integrally convex.
Both $x$ and $y$ are dec-min in $S$, but we have $\| (2,1,0,0) - (0,0,1,2)  \|_{\infty}=2$. 
Note that the convex hull $\overline{S}$ (line segment connecting $x$ and $y$) 
is an integral polyhedron and $\overline{S} \cap \ZZ\sp{4}=S$.
\finbox
\end{example}

The objective of this paper is to give two different proofs to 
Theorem~\ref{THdecminICunit}.
The first proof, given in Section~\ref{SCproofFenc}, resorts to 
the Fenchel-type duality theorem 
for integrally convex functions
established recently by the present authors \cite{MT22ICfenc}.
Although this seems to be a nice application of a recent result
in discrete convex analysis,
it is also natural to ask whether such a machinery is inevitable 
to prove Theorem~\ref{THdecminICunit}.
For example, the original proofs 
\cite{FM22partA,FM22fairsbmflow}
of Theorems \ref{THdecminMunit} and \ref{THdecminM2unit}
for M-convex and \MM-convex cases
are based on standard tools in combinatorial optimization.
In Section~\ref{SCproofAlt},
we give the second elementary proof,
which shows the existence of a box $[z, z']_{\ZZ}$
in the statement without relying on any tools and 
the existence of a face $F$ on the basis of Farkas' lemma.


\section{Preliminaries}
\label{SCprelim}

\subsection{Integrally convex sets}
\label{SCintcnv}

The concept of integrally convex functions
was introduced by Favati and Tardella \cite{FT90}
and its set version was formulated 
in \cite[Sec.~3.4]{Mdcasiam}.

For any $x \in \RR^{n}$ the
{\em integral neighborhood}
of $x$ is defined by
\begin{equation}  \label{intneighbordef}
{\rm N}(x) = \{ z \in \ZZ^{n} \mid | x_{i} - z_{i} | < 1 \ (i=1,2,\ldots,n)  \} .
\end{equation}
It is noted that 
strict inequality 
``\,$<$\,'' 
is used in this definition,
so that ${\rm N}(x)$ admits an alternative expression
\begin{equation}  \label{intneighbordeffloorceil}
{\rm N}(x) = \{ z \in \ZZ\sp{n} \mid
\lfloor x_{i} \rfloor \leq  z_{i} \leq \lceil x_{i} \rceil  \ \ (i=1,2,\ldots, n) \} ,
\end{equation}
where, for $t \in \RR$ in general, 
$\left\lfloor  t  \right\rfloor$
denotes the largest integer not larger than $t$
({\em rounding-down} to the nearest integer)
and 
$\left\lceil  t   \right\rceil$ 
is the smallest integer not smaller than $t$
({\em rounding-up} to the nearest integer).
That is,
${\rm N}(x)$ consists of all integer vectors $z$ 
between
$\left\lfloor x \right\rfloor 
=( \left\lfloor x_{1} \right\rfloor ,\left\lfloor x_{2} \right\rfloor , 
  \ldots, \left\lfloor x_{n} \right\rfloor)$ 
and 
$\left\lceil x \right\rceil 
= ( \left\lceil x_{1} \right\rceil, \left\lceil x_{2} \right\rceil,
   \ldots, \left\lceil x_{n} \right\rceil)$.

For any set $S \subseteq \ZZ^{n}$
and $x \in \RR^{n}$
we call the convex hull of $S \cap {\rm N}(x)$ 
the 
{\em local convex hull}
of $S$ around $x$.
A nonempty set
$S \subseteq \ZZ^{n}$ is said to be 
{\em integrally convex}
if the union of the local convex hulls $\overline{S \cap {\rm N}(x)}$ over $x \in \RR^{n}$ 
coincides with the convex hull of $S$,
that is, if
\begin{equation}  \label{icsetdef0}
 \overline{S} = \bigcup_{x \in \RR\sp{n}} \overline{S \cap {\rm N}(x)}.
\end{equation}
Condition \eqref{icsetdef0} is equivalent to the following:
\begin{align}  
& x \in \overline{S} \ \Longrightarrow  \  x \in  \overline{S \cap {\rm N}(x)} 
\qquad
\mbox{for all $x \in \RR\sp{n}$}.
\label{icsetdef1}
\end{align}
It is pointed out in \cite{Mopernet21} that 
an integrally convex set is precisely the set of integer points
of a box-integer polyhedron (see \cite[Section~5.15]{Sch03}
for the definition of a box-integer polyhedron).
Obviously, 
every subset of $\{ 0, 1\}\sp{n}$ is an integrally convex set. 
It is known \cite{Mdcasiam} that M-convex sets and \MM-convex sets
are integrally convex.

\subsection{Convex characterization of dec-minimality}

We describe a convex characterization of dec-minimality.

A function 
$\varphi: \ZZ \to \Rinf$
in a single integer variable
is called  a (univariate) 
{\em discrete convex function}
if its effective domain $\dom \varphi =\{ k \in \ZZ \mid \varphi(k) < +\infty \}$ 
is an interval of integers
and the inequality
$ 
\varphi(k-1) + \varphi(k+1) \geq 2 \varphi(k)
$ 
holds for each $k \in \dom \varphi$.
A function 
$\Phi: \ZZ^{n} \to \Rinf$
is called a
{\em separable convex function}
if it can be represented as
$\Phi(x) = \sum_{i=1}\sp{n} \varphi_{i}(x_{i})$
with univariate discrete convex functions $\varphi_{i}: \ZZ \to \Rinf$
$(i=1,2,\ldots,n)$.

Consider a  symmetric separable convex function 
$\Phi_{\rm sym}(x) = \sum_{i=1}\sp{n} \varphi(x_{i})$
defined by a finite-valued convex function $\varphi: \ZZ \to \RR$.
If
\begin{equation} \label{rapidinc}
 \varphi(k+1) \geq  n \cdot  \varphi(k)  > 0 
\qquad (k \in \ZZ) ,
\end{equation}
we say that $\Phi_{\rm sym}$ is a
{\em rapidly increasing}
symmetric separable convex function
(with respect to $n$),
and use notation $\Phi_{\rm rap}$
with the subscript ``{\rm rap}'' for ``rapidly increasing.''
That is,
\[
\Phi_{\rm rap}(x) = \sum_{i=1}\sp{n} \varphi(x_{i})
\]
denotes a symmetric separable convex function
defined by a positive-valued discrete convex function 
$\varphi: \ZZ \to \RR$
satisfying \eqref{rapidinc}.
When $n \geq 2$,
such $\varphi$ is strictly convex,
since 
\[
\varphi(k-1) +  \varphi(k+1) > \varphi(k+1) \geq n \varphi(k) \geq 2 \varphi(k)
\]
for all $k \in \ZZ$.
Note that
$\varphi(k) \to +\infty$ as $k \to +\infty$ and
$\varphi(k) \to 0$ as $k \to -\infty$.
For example,  
$\varphi (k) = c\sp{k}$ with $c \ge n$
satisfies \eqref{rapidinc}.

The following theorem shows a fundamental fact, 
the equivalence between decreasing minimization 
and minimization of the function $\Phi_{\rm rap}$.
We denote the set of all minimizers of
$\Phi_{\rm rap}$ on $S$
by $\argmin (\Phi_{\rm rap} | S)$.
No assumption is made on the type of discrete convexity
(like M-convexity) of the set $S \subseteq \ZZ\sp{n}$.

\begin{theorem}[{\cite{FM19partII}, \cite[Theorem~13.5]{Mdcamarz24}}]  \label{THdecminRapid}
Assume $n \geq 2$. Let $S \subseteq \ZZ\sp{n}$ and 
$\Phi_{\rm rap}(x) = \sum_{i=1}\sp{n} \varphi(x_{i})$
be a symmetric separable convex function
with $\varphi$ satisfying \eqref{rapidinc}.

\noindent
{\rm (1)} \ 
For any $x, y \in \ZZ\sp{n}$, we have:
$x <_{\rm dec} y$ $\iff$
$\Phi_{\rm rap}(x) < \Phi_{\rm rap}(y)$. 

\noindent
{\rm (2)} \
$S$ has a dec-min element
$\iff$ 
$\Phi_{\rm rap}$ has a minimizer on $S$.

\noindent
{\rm (3)} \ 
$\decmin(S) = \argmin (\Phi_{\rm rap} | S)$.
\end{theorem}


\section{Proof Based on Fenchel-type Duality}
\label{SCproofFenc}

In this section we show the proof of Theorem~\ref{THdecminICunit}
based on the Fenchel-type duality theorem
in discrete convex analysis.
While this proof was sketched in \cite[Section~13.7]{Mdcamarz24} 
(in Japanese), 
we describe it here in more detail.

We rely on the Fenchel-type duality theorem 
for integrally convex functions
\cite[Theorem~1.1]{MT22ICfenc}.
The theorem, specialized to separable convex minimization on an integrally convex set,
reads as follows.  We use notations:
\begin{align*}
& \langle p, x \rangle 
= p_{1} x_{1} + p_{2} x_{2} +  \cdots + p_{n} x_{n}
 \qquad 
(p \in \RR\sp{n}, x \in \ZZ\sp{n}),
\\ &
\Phi[-p](x) 
= \Phi(x) - \langle p , x \rangle  
= \Phi(x) - \sum_{i=1}\sp{n} p_{i} x_{i}
 \qquad 
(p \in \RR\sp{n}, x \in \ZZ\sp{n}),
\\ &
\argmin \Phi[-p] = \{ x \in \ZZ\sp{n} \mid 
   \Phi[-p](x) \leq \Phi[-p](y) \mbox{ for all $y \in \ZZ\sp{n}$} \}
\qquad (p \in \RR\sp{n}),
\\ &
\theta_{S}(p) = \inf \{ \langle p, x \rangle \mid  x\in S \}
 \qquad (p \in \RR\sp{n}),
\\ &
\Phi\sp{\bullet}(p) = \sup \{ \langle p, x \rangle - \Phi(x) \mid x \in \ZZ\sp{n} \}
 \qquad (p \in \RR\sp{n}).
\end{align*}

\begin{theorem}[{\cite[Theorem~6.2]{MT23ICsurv}}]   \label{THfencICsetsep}
Let
$S \ (\subseteq \ZZ\sp{n})$
be an integrally convex set and 
$\Phi: \ZZ\sp{n} \to \Rinf$
a separable convex function.
Assume that
$S \cap \dom \Phi \neq \emptyset$ and
$\inf \{ \Phi(x) \mid  x \in S  \}$
is attained by some $x \in S \cap \dom \Phi$.
Then
\[
\min \{ \Phi(x) \mid  x \in S  \} 
= \max \{ \theta_{S}(p) - \Phi\sp{\bullet}(p) \mid  p \in \RR\sp{n} \},
\]
where the maximum is attained by some $p = p\sp{*} \in \RR\sp{n}$. 
Moreover, the set of all primal optimal solutions can be described as
\begin{equation} \label{fnsumargminICsep}
\argmin (\Phi | S)
 = \argmin_{x \in S} ( \langle p\sp{*}, x \rangle )
 \cap \argmin (\Phi[-p\sp{*}])
\end{equation} 
with an arbitrary dual optimal solution $p\sp{*}$.
If, in addition, $\Phi$ is integer-valued, then
\begin{equation*} 
 \min \{ \Phi(x) \mid  x \in S  \} 
= \max \{ \theta_{S}(p) - \Phi\sp{\bullet}(p) \mid  p \in \ZZ\sp{n} \} ,
\end{equation*} 
where the maximum is attained by some $p \in \ZZ\sp{n}$. 
\end{theorem}

The following proposition gives a key identity \eqref{fnsumargminIC6}.

\begin{proposition}     \label{PRdecminICpexist}
Assume $n \ge 2$.
Let $S \ (\subseteq \ZZ\sp{n})$
be an integrally convex set admitting a dec-min element
and 
$\Phi_{\rm rap}(x) = \sum_{i=1}\sp{n} \varphi(x_{i})$
be a rapidly increasing symmetric separable convex function
with $\varphi$ satisfying \eqref{rapidinc}.
Then
\begin{equation} \label{fnsumargminIC6}
\decmin(S) =
\argmin (\Phi_{\rm rap} | S)
 = \argmin_{x \in S} ( \langle p\sp{*}, x \rangle )
 \cap \argmin (\Phi_{\rm rap}[-p\sp{*}])
\end{equation} 
for some $p\sp{*} \in \RR\sp{n}$.
\end{proposition}

\begin{proof}
By Theorem~\ref{THdecminRapid}
we have 
\begin{equation}  \label{decminargminrapid}
 \decmin(S) =\argmin (\Phi_{\rm rap} | S),
\end{equation}
which is nonempty by the assumption that  $S$ admits a dec-min element.
On the other hand, the Fenchel-type duality 
(Theorem~\ref{THfencICsetsep}) shows that
\[ 
 \min \{ \Phi_{\rm rap}(x) \mid  x \in S  \} 
= \max \{ \theta_{S}(p) - \Phi_{\rm rap}\sp{\bullet}(p) \mid  p \in \RR\sp{n} \},
\] 
where the minimum is attained.
Therefore, the maximum is attained by some $p=p\sp{*}$, for which 
we have
\begin{equation} \label{fnsumargminICseprap}
\argmin (\Phi_{\rm rap} | S)
 = \argmin_{x \in S} ( \langle p\sp{*}, x \rangle )
 \cap \argmin (\Phi_{\rm rap}[-p\sp{*}])
\end{equation} 
by \eqref{fnsumargminICsep}.
A combination of 
\eqref{decminargminrapid} and \eqref{fnsumargminICseprap} yields
\eqref{fnsumargminIC6}.
\end{proof}

To prove Theorem~\ref{THdecminICunit}, we may assume $n \ge 2$
(the case of $n = 1$ is trivial).
Using the vector  $p\sp{*}$ in Proposition~\ref{PRdecminICpexist},
define
$\displaystyle  F =  \argmin_{x \in \overline{S}} ( \langle p\sp{*}, x \rangle)$.
This is a face of $\overline{S}$
and we have
\[
\displaystyle  \argmin_{x \in S} ( \langle p\sp{*}, x \rangle) = F \cap \ZZ\sp{n},
\]
whereas by strict convexity of $\varphi$, we have
\[
\argmin (\Phi_{\rm rap}[-p\sp{*}]) = [z, z']_{\ZZ}
\]
for some $z,  z' \in \ZZ\sp{n}$ 
with $\veczero \le z' - z \le \vecone$.
Substituting these into 
\eqref{fnsumargminIC6}, we obtain
\[
\decmin(S) = \argmin (\Phi_{\rm rap} | S) 
= F \cap [z, z']_{\ZZ} \subseteq z + \{ 0,1 \}\sp{n}.
\]
Any set contained in a unit cube is integrally convex.
Hence follows the integral convexity of $\decmin(S)$.

\begin{remark} \rm  \label{RMproofMM2}
Theorems \ref{THdecminMunit} and \ref{THdecminM2unit}
can be proved in the same way as above with some additional observations.
First assume that $S$ is an M-convex set.
All vectors in $S$ have the same component-sum,
which guarantees the existence of a dec-min element,
i.e., $\decmin(S) \ne \emptyset$.
The set $\argmin_{x \in S} ( \langle p\sp{*}, x \rangle )$
is M-convex,
and the intersection of this M-convex set with
$[z, z']_{\ZZ}$ is also M-convex.
Thus we obtain Theorem~\ref{THdecminMunit}.
The proof of Theorem~\ref{THdecminM2unit} for an \MM-convex set 
can be obtained with the observation that
$\argmin_{x \in S} ( \langle p\sp{*}, x \rangle )$
is an \MM-convex set and the intersection of an \MM-convex set with
$[z, z']_{\ZZ}$ is also \MM-convex.
\finbox
\end{remark}

\begin{example} \rm \label{EXdecminICmaru1315fenc}
The formula \eqref{fnsumargminIC6} is illustrated for a simple example
\cite[Example~13.15]{Mdcamarz24}.
In particular, we identify the vector $p\sp{*}$ in the formula.

\begin{enumerate}
\item
Let
$S= \{ 
(2,1,1,0), 
(2,1,0,1), 
(1,2,1,0), 
(1,2,0,1), 
(2,2,0,0)
 \}$,
which is an integrally convex set
(actually an M-convex set).
Putting
$x\sp{1}=(2,1,1,0)$, 
$x\sp{2}=(2,1,0,1)$,
$x\sp{3}=(1,2,1,0)$, 
$x\sp{4}=(1,2,0,1)$, and 
$x\sp{5}=(2,2,0,0)$,
we have
$S = \{ x\sp{1}, x\sp{2},x\sp{3}, x\sp{4}, x\sp{5} \}$
and
\[
\decmin(S) = \{ 
(2,1,1,0), 
(2,1,0,1), 
(1,2,1,0), 
(1,2,0,1)
 \} 
= \{ x\sp{1}, x\sp{2},x\sp{3}, x\sp{4} \}.
\]

\item
Let 
$\Phi_{\rm rap}(x) = \sum_{i=1}\sp{4} \varphi(x_{i})$
with
$\varphi(k) = 10\sp{k}$.
We have
\begin{align*}
& \Phi_{\rm rap}(x\sp{j}) = \Phi_{\rm rap}((2,1,1,0)) =
 \varphi(2) + 2 \varphi(1) +  \varphi(0) = 121
\quad (j=1,2,3,4),
\\ &
\Phi_{\rm rap}(x\sp{5})= \Phi_{\rm rap}((2,2,0,0))
 = 2 \varphi(2) + 2 \varphi(0)  = 202 .
\end{align*}
We have $\decmin(S) =\argmin (\Phi_{\rm rap} | S)$
as in Theorem~\ref{THdecminRapid}(3).

\item
The condition on $p\sp{*}$ for the inclusion 
$\displaystyle \argmin_{x \in S} ( \langle p\sp{*}, x \rangle ) \supseteq \decmin(S)$
is given by
\begin{align*}
& 
 2 p\sp{*}_{1} + p\sp{*}_{2} + p\sp{*}_{3} =  
 2 p\sp{*}_{1} + p\sp{*}_{2} + p\sp{*}_{4} =
  p\sp{*}_{1} + 2 p\sp{*}_{2} + p\sp{*}_{3} =
  p\sp{*}_{1} + 2 p\sp{*}_{2} + p\sp{*}_{4} \le
  2 p\sp{*}_{1} + 2  p\sp{*}_{2}
\nonumber
 \\ & \iff
p\sp{*}_{1} = p\sp{*}_{2} \ge p\sp{*}_{3} = p\sp{*}_{4}
\iff
 p\sp{*}= (a, a, b, b)
\quad (a \ge b).
\nonumber
\end{align*}
More precisely, 
for $p\sp{*}= (a, a, b, b)$, 
we have
\begin{align}
\argmin_{x \in S} ( \langle p\sp{*}, x \rangle ) =
\begin{cases}
\{ x\sp{1}, x\sp{2},x\sp{3}, x\sp{4} \} & (a > b),
\\
\{ x\sp{1}, x\sp{2},x\sp{3}, x\sp{4}, x\sp{5} \} & (a = b).
\end{cases}
\label{Bpab2}
\end{align}

\item
On the other hand, the condition on $p\sp{*}$ for the inclusion 
$\displaystyle \argmin (\Phi_{\rm rap}[-p\sp{*}]) \supseteq \decmin(S)$
is given as follows.
By the optimality criterion for 
(unconstrained) 
minimization of a separable convex function,
we have
\begin{align*}
& x\sp{1}=(2,1,1,0) \in \argmin (\Phi_{\rm rap}[-p\sp{*}])
\\ &
\iff
\varphi(2) - \varphi(1) \le p\sp{*}_{1} \le \varphi(3) - \varphi(2), \
\\ & \phantom{\ \iff \ }
\varphi(1) - \varphi(0) \le p\sp{*}_{i} \le \varphi(2) - \varphi(1) \ \ (i=2,3), \ 
\\ & \phantom{\ \iff \ }
\varphi(0) - \varphi(-1) \le p\sp{*}_{4} \le \varphi(1) - \varphi(0)
\\ &
\iff
90 \le p\sp{*}_{1} \le 900, \ 
9 \le p\sp{*}_{2} \le 90, \ 
9 \le p\sp{*}_{3} \le 90, \ 
0.9 \le p\sp{*}_{4} \le 9;
\\
& x\sp{2}=(2,1,0,1) \in \argmin (\Phi_{\rm rap}[-p\sp{*}])
\\ &
\iff
90 \le p\sp{*}_{1} \le 900, \ 
9 \le p\sp{*}_{2} \le 90, \ 
0.9 \le p\sp{*}_{3} \le 9, \ 
9 \le p\sp{*}_{4} \le 90;
\\
& x\sp{3}=(1,2,1,0) \in \argmin (\Phi_{\rm rap}[-p\sp{*}])
\\ &
\iff
9 \le p\sp{*}_{1} \le 90, \ 
90 \le p\sp{*}_{2} \le 900, \ 
9 \le p\sp{*}_{3} \le 90, \ 
0.9 \le p\sp{*}_{4} \le 9;
\\
& x\sp{4}=(1,2,0,1) \in \argmin (\Phi_{\rm rap}[-p\sp{*}])
\\ &
\iff
9 \le p\sp{*}_{1} \le 90, \ 
90 \le p\sp{*}_{2} \le 900, \ 
0.9 \le p\sp{*}_{3} \le 9, \ 
9 \le p\sp{*}_{4} \le 90.
\end{align*}
Therefore,
\begin{align}
& \argmin (\Phi_{\rm rap}[-p\sp{*}]) \supseteq \decmin(S)
\nonumber \\ & 
\iff
p\sp{*} = (90, 90, 9, 9)
\iff
\argmin (\Phi_{\rm rap}[-p\sp{*}]) = B\sp{\circ},
\label{BpPhi0}
\end{align}
where
$B\sp{\circ}= [(1,1,0,0), (2,2,1,1)]_{\ZZ} 
 = \{ 1,2 \}\sp{2} \times \{ 0,1 \}\sp{2}$.

\item
By \eqref{Bpab2} and \eqref{BpPhi0}, we uniquely obtain
$p\sp{*} = (90, 90, 9, 9)$, with which 
the formula \eqref{fnsumargminIC6} holds with
\begin{align*}
 \argmin_{x \in S} ( \langle p\sp{*}, x \rangle ) &=
\{ x\sp{1}, x\sp{2},x\sp{3}, x\sp{4} \},
\\ 
\argmin (\Phi_{\rm rap}[-p\sp{*}]) &= B\sp{\circ}.
\end{align*}
Since
$B\sp{\circ} \supseteq S$, \eqref{fnsumargminIC6} reduces to
$\decmin(S) = \argmin_{x \in S} ( \langle p\sp{*}, x \rangle )$
in this example.
\finbox
\end{enumerate}
\end{example}

\section{Alternative Elementary Proof}
\label{SCproofAlt}

\subsection{Bounding by a unit box}
\label{SCproofAltBox}

In this section we prove the following statement,
which constitutes a part of Theorem~\ref{THdecminICunit}.

\begin{proposition}     \label{PRdecminICunit}
Let $S$ be an integrally convex set admitting a dec-min element.
For any $x, y \in \decmin (S)$, we have
$\| x - y \|_{\infty} \le 1$.
\end{proposition}

Recall $N= \{ 1,2,\ldots, n \}$.
We sometimes use notation
$\RR\sp{N}$ for $\RR\sp{n}$ 
to emphasize the ground set $N$.
We first prepare a general lemma,
where $\supp (z) = \{ i \in N \mid z_{i} \ne 0 \}$
for any $z \in \RR\sp{N}$.

\begin{lemma}    \label{LMcube}
Let $T \subseteq \{ 0, 1 \}\sp{N}$ and
$U \subseteq N$,
and let $y$ be a convex combination of elements of $T$, that is,
\begin{equation} \label{cl1eq1}
y = \sum_{z \in T} \lambda_{z} z,
\qquad \sum_{z \in T} \lambda_{z} = 1,
\quad \lambda_{z} > 0 \  \ (\forall z \in T).
\end{equation}
If $y_{i} = 1/2$ for all $i \in U$, 
there exists $\hat z  \in T$ satisfying 
$| \supp (\hat z) \cap U | \le \left\lfloor {|U|}/{2} \right\rfloor $.
\end{lemma}
\begin{proof}
To prove the claim by contradiction, assume that
$| \supp (z) \cap U | \ge \left\lfloor {|U|}/{2} \right\rfloor + 1 $
for every $z  \in T$.
By considering the sum of the components within $U$,
we obtain
\[
y(U) = \sum_{z \in T} \lambda_{z} z(U) .
\]
By the assumption we have $y(U) = \sum_{i \in U} y_{i} = |U|/2$, whereas
\begin{align*} 
 \sum_{z \in T} \lambda_{z} z(U)
= \sum_{z \in T} \lambda_{z} | \supp (z) \cap U |
\ge  \sum_{z \in T} \lambda_{z} ( \left\lfloor {|U|}/{2} \right\rfloor + 1) 
= \left\lfloor {|U|}/{2} \right\rfloor + 1 > |U|/2 .
\end{align*} 
This is a contradiction, proving that
$| \supp (\hat z) \cap U | \le \left\lfloor {|U|}/{2} \right\rfloor $
for some $\hat z  \in T$.
\end{proof}

To prove Proposition~\ref{PRdecminICunit} by contradiction,
suppose that 
there exist $x, y \in \decmin (S)$
with $\| x - y \|_{\infty} \ge 2$.

We first reduce our argument to the case where
$x_{i} \ne y_{i}$ for all $i \in N$.
Let $N' = \{ i \in N \mid x_i = y_i \}$. 
If $N' \ne \emptyset$,
let 
$\hat N = N \setminus N'$ and consider 
\[
\hat{S} = \{ z|_{\hat{N}} \mid z \in S, z|_{N'} = x|_{N'} \},
\qquad
\hat x = x|_{\hat N}, \qquad
\hat y = y|_{\hat N},
\]
where,
for any vector $z$ on $N$ and any subset $U$ of $N$,
$z|_{U}$ denotes the restriction of $z$ to $U$.
Then 
$\hat S$ is an integrally convex set,
$\hat x, \hat y \in \decmin(\hat S)$,
$\| \hat x - \hat y \|_{\infty} \ge 2$,
and 
$\hat x_{i} \ne \hat y_{i}$ for all $i \in \hat N$.
In the following 
we denote
$(\hat S, \hat x, \hat y)$
simply by
$(S, x, y)$
and assume that
\begin{description}
\item
[\rm (A0)]
$x_{i} \ne y_{i}$ for all $i \in N$.
\end{description}

Let
\begin{equation} \label{boxprDdef}
D = \{ i \in N \mid | x_{i} - y_{i} | \ge 2 \},
\end{equation}
which is nonempty.
Let $p$ denote the index $i \in D$
at which 
$\max (x_{i}, y_{i})$
is maximized over $D$, that is,
\begin{equation} \label{pargmax}
p \in \argmax \{  \max (x_{i}, y_{i}) \mid i \in D \}.
\end{equation}
By interchanging $x$ and $y$ if necessary,
we may assume that
\begin{description}
\item
[\rm (A1)]
$x_{p} = \max \{  x_{i} \mid i \in D \} \ge \max \{ y_{i} \mid i \in D \}$.
\end{description}
Then we have
\begin{description}
\item
[\rm (A2)]
$x_{p} \ge y_{p}+2$.
\end{description}

Let
$\alpha_{1} > \alpha_{2} > \cdots > \alpha_{r}$
denote the distinct values of the components of $x$
and define
\begin{equation} \label{Vjdef}
V_{j} = \{ i \in N \mid x_{i} = \alpha_{j} \}
\qquad
(j=1,2,\ldots,r).
\end{equation}
These subsets are pairwise disjoint and define a partition 
$N = V_{1} \cup V_{2} \cup \cdots \cup V_{r}$;
see Fig.~\ref{FGstairimage}.

\begin{figure}
\centering
\includegraphics[width=0.45\textwidth,clip]{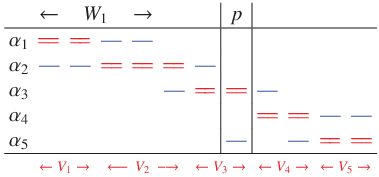}
\caption{Partition $\{ V_{j} \}$ of $N$ 
 ($x$: \RED{=\!=}, $y$: \BLU{---})}
\label{FGstairimage}
\end{figure}

Let
\begin{equation} \label{zdef}
 z = \frac{x+y}{2}.
\end{equation}
By the integral convexity of $S$,
we can represent $z$ as a convex combination of the elements of 
$S \cap {\rm N}(z)$.
That is,
\begin{equation} \label{zconvcombi}
z = \sum_{k=1}\sp{K} \lambda_{k} z\sp{k},
\qquad \sum_{k=1}\sp{K} \lambda_{k}  = 1,
\quad \lambda_{k} > 0 \ \ (k=1,2,\ldots, K)
\end{equation}
for some $\{ z\sp{1}, z\sp{2}, \ldots, z\sp{K} \} \subseteq S \cap {\rm N}(z)$.
Note that
$x \notin {\rm N}(z)$ and
$y \notin {\rm N}(z)$
since
$|x_{p} - y_{p}| \ge 2$.

\medskip

{\bf Step 1:}
As the first step, we consider 
the indices $i \in N$ 
with $\{ x_{i}, y_{i} \} = \{ \alpha_{1}, \alpha_{1}-1 \}$.
Let
\begin{equation} \label{W1def}
W_{1} = 
\{ i \in N \mid x_{i} = \alpha_{1}, \  y_{i} = \alpha_{1}-1 \}
\cup \{ i \in N \mid x_{i} = \alpha_{1}-1,  \  y_{i} = \alpha_{1}\};
\end{equation}
see Fig.~\ref{FGstairimage}.
We may have $W_{1} = \emptyset$. 
It follows from (A0) and the definition \eqref{zdef} of $z$ that
\begin{align} 
& 
z\sp{k}_{i} \le \alpha_{1}
\ \ \quad \qquad(i \in W_{1}; k=1,2,\ldots,K),
\label{zkalp1}
\\ &
z\sp{k}_{i} \le \alpha_{1}-1
\qquad(i \in N \setminus W_{1}; k=1,2,\ldots,K).
\label{zkalp11}
\end{align}

\begin{claim}    \label{CLbox1}
$p \in N \setminus V_{1}$. 
\end{claim}

\begin{proof}
To prove by contradiction, suppose that 
$p \in V_{1}$.
Since
$x{\downarrow} = y{\downarrow}$,
we have
$| \{ i \in N \mid x_{i} = \alpha_{1} \}|
 = | \{ i \in N \mid y_{i} = \alpha_{1} \}| = |V_{1}|$,
whereas $p \notin W_{1}$ 
from (A2). Therefore,
\begin{equation} \label{W1small}
|W_{1}| \le (|V_{1}| -1) +  |V_{1}| \le 2 |V_{1}| -1 .
\end{equation}
Let $c = \left\lfloor z \right\rfloor$
and consider
$T = \{ z\sp{1} -c , z\sp{2} -c, \ldots, z\sp{K} -c \} \subseteq \{ 0, 1 \}\sp{N}$.
Since
$z_{i}-c_{i} = 1/2$ for all $i \in W_{1}$,
we can apply Lemma~\ref{LMcube}
to obtain an index $k\sp{*}$ such that
\begin{equation} \label{step1byLemma}
 | \{ i \in N \mid z\sp{k\sp{*}}_{i} = \alpha_{1} \}|
= | \supp (z\sp{k\sp{*}}-c) \cap W_{1} | 
\le \left\lfloor |W_{1}|/2 \right\rfloor .
\end{equation}
It follows from \eqref{W1small} and \eqref{step1byLemma} that
\[
 | \{ i \in N \mid z\sp{k\sp{*}}_{i} = \alpha_{1} \}|
\le \left\lfloor |W_{1}|/2 \right\rfloor 
\le \left\lfloor |V_{1}| -1/2 \right\rfloor 
= |V_{1}| -1 ,
\]
which,
together with \eqref{zkalp11},
shows $z\sp{k\sp{*}} <_{\rm dec} x$,
a contradiction to the dec-minimality of $x$ in $S$.
\end{proof}

\begin{claim}    \label{CLbox1A}
$|W_{1}| = 2 |V_{1}|$.
\end{claim}
\begin{proof}
Recall from \eqref{W1def} that $W_{1}$ consists of two parts.
For the first part, we have
\begin{equation}   \label{boxW1part1}
\{ i \in N \mid x_{i} = \alpha_{1}, \  y_{i} = \alpha_{1}-1 \} 
=\{ i \in N \mid x_{i} = \alpha_{1} \} =V_{1}
\end{equation}
since
$p \notin V_{1}$
by Claim~\ref{CLbox1}.
For the second part, we have
\begin{equation}   \label{boxW1part2}
 \{ i \in N \mid x_{i} = \alpha_{1}-1,  \  y_{i} = \alpha_{1}\} 
=  \{ i \in N \mid y_{i} = \alpha_{1} \} ,
\quad
 | \{ i \in N \mid y_{i} = \alpha_{1} \} | 
= |V_{1}| ,
\end{equation}
since
$\{ i \in N \mid x_{i} \le \alpha_{1}-2,  \  y_{i} = \alpha_{1}\} = \emptyset$
from the definition \eqref{pargmax} of $p$ and 
the assumption (A1).
Hence follows $|W_{1}| = 2 |V_{1}|$. 
\end{proof}

Note that 
$V_{1} \subseteq W_{1}$ and $p \notin W_{1}$
follow from 
Claims \ref{CLbox1} and \ref{CLbox1A}.

\begin{claim}    \label{CLbox1B}
For vectors $z\sp{k}$ in \eqref{zconvcombi}, we have
\begin{equation}  \label{TmemoEq7}
| \{ i \in W_{1}  \mid z\sp{k}_{i} = \alpha_{1} \}| 
=
| \{ i \in N \mid z\sp{k}_{i} = \alpha_{1} \}| \ge | V_{1}|
\qquad (k=1,2,\ldots,K).
\end{equation}
\end{claim}
\begin{proof}
The equality follows from 
\eqref{zkalp1} and \eqref{zkalp11}. 
Since
$z\sp{k}_{i} \leq  \alpha_{1}$
for all $i \in N$,
the inequality
$| \{ i \in N \mid z\sp{k}_{i} = \alpha_{1} \}| < | V_{1}|$
would imply $z\sp{k} <_{\rm dec} x$,
contradicting the dec-minimality of $x$.
\end{proof}

Recall that, for any vector $z$ on $N$, $z|_{W_{1}}$ 
denotes the restriction of $z$ to $W_{1}$.

\begin{claim}    \label{CLbox1C}
$(x|_{W_{1}}){\downarrow} = (y |_{W_{1}}){\downarrow}= 
(z\sp{k}|_{W_{1}}){\downarrow}$
\quad
$(k=1,2,\ldots,K)$. 
\end{claim}

\begin{proof}
Since
$z_{i} = \alpha_{1} - 1/2$ for all $i \in W_{1}$
and $|W_{1}| = 2 |V_{1}|$ 
by Claim \ref{CLbox1A},
we have
\begin{align*}
 z(W_{1}) &= |W_{1}| (\alpha_{1} - 1/2) = (2 \alpha_{1} -1) |V_{1}| .
\end{align*}
On the other hand, it follows from
\eqref{zconvcombi} and
Claims \ref{CLbox1A} and \ref{CLbox1B}
that
\begin{align*}
 z(W_{1}) &= 
\sum_{k=1}\sp{K} \lambda_{k} z\sp{k}(W_{1})
\ge \sum_{k=1}\sp{K} \lambda_{k} (\alpha_{1} |V_{1}| + (\alpha_{1}-1) |V_{1}|)
= (2 \alpha_{1} -1 ) |V_{1}|.
\end{align*}
Therefore,
equality holds in \eqref{TmemoEq7},
that is,
\begin{equation*} 
 | \{ i \in W_{1} \mid z\sp{k}_{i} = \alpha_{1} \}|
= | \{ i \in W_{1} \mid z\sp{k}_{i} = \alpha_{1} -1 \}|
= | V_{1}|
\qquad (k=1,2,\ldots,K).
\end{equation*} 
Combining this with
\eqref{boxW1part1} and \eqref{boxW1part2},
we obtain
$(x|_{W_{1}}){\downarrow} = (y |_{W_{1}}){\downarrow}= 
(z\sp{k}|_{W_{1}}){\downarrow}$.
\end{proof}

By Claim~\ref{CLbox1C}, we may concentrate on 
the components of the vectors
within $N \setminus W_{1}$.

\medskip

{\bf Step 2:}
As the second step, we consider 
the indices $i \in N \setminus W_{1}$
with $\{ x_{i}, y_{i} \} = \{ \alpha_{2}, \alpha_{2}-1 \}$.
Define
$N' = N \setminus W_{1}$ and
\begin{equation} \label{W2def}
W_{2} = 
\{ i \in N' \mid x_{i} = \alpha_{2}, \  y_{i} = \alpha_{2}-1 \}
\cup \{ i \in  N' \mid x_{i} = \alpha_{2}-1,  \  y_{i} = \alpha_{2}\}.
\end{equation}
Similarly to 
\eqref{zkalp1} and \eqref{zkalp11}, we have
\begin{align} 
& 
z\sp{k}_{i} \le \alpha_{2}
\ \ \quad \qquad(i \in W_{2}; k=1,2,\ldots,K),
\label{zkalp2}
\\ &
z\sp{k}_{i} \le \alpha_{2}-1
\qquad(i \in N' \setminus W_{2}; k=1,2,\ldots,K).
\label{zkalp21}
\end{align}

\begin{claim}    \label{CLbox2}
$p \in N' \setminus V_{2}$. 
\end{claim}

\begin{proof}
We have $p \notin W_{1}$ as we noted before Claim~\ref{CLbox1B}. 
Suppose that $p \in N' \cap V_{2}$,
from which we intend to derive a contradiction.
By $x{\downarrow} = y{\downarrow}$ and
$(x|_{W_{1}}){\downarrow} = (y |_{W_{1}}){\downarrow}$ 
in Claim~\ref{CLbox1C},
we have
$| \{ i \in N' \mid x_{i} = \alpha_{2} \}|
 = | \{ i \in N' \mid y_{i} = \alpha_{2} \}| = |N' \cap V_{2}|$,
whereas
$p \notin W_{2}$ 
from (A2). Therefore,
\begin{equation} \label{W2small}
|W_{2}| \le (|N' \cap V_{2}| -1) +  |N' \cap V_{2}| \le 2 |N' \cap V_{2}| -1 .
\end{equation}
Consider
$T' = \{ (z\sp{k} -c)|_{N'} \mid k=1,2,\ldots, K \} \subseteq \{ 0, 1 \}\sp{N'}$,
where $c = \left\lfloor z \right\rfloor$.
Since
$z_{i}-c_{i} = 1/2$ for all $i \in W_{2}$,
we can apply Lemma~\ref{LMcube}
to obtain an index $k\sp{*}$ such that
\begin{equation} \label{step2byLemma}
| \{ i \in N' \mid z\sp{k\sp{*}}_{i} = \alpha_{2} \}|
 = | \supp ((z\sp{k\sp{*}}-c)|_{N'}) \cap W_{2} | 
\le \left\lfloor |W_{2}|/2 \right\rfloor .
\end{equation}
It follows from \eqref{W2small} and \eqref{step2byLemma} that
\[
| \{ i \in N' \mid z\sp{k\sp{*}}_{i} = \alpha_{2} \}|
\le \left\lfloor |W_{2}|/2 \right\rfloor 
\le \left\lfloor |N' \cap V_{2}| -1/2 \right\rfloor  
= |N' \cap V_{2}| -1 .
\]
Combining this with 
$(x|_{W_{1}}){\downarrow} = (z\sp{k\sp{*}}|_{W_{1}}){\downarrow}$
in Claim~\ref{CLbox1C}
and \eqref{zkalp21},
we obtain $z\sp{k\sp{*}} <_{\rm dec} x$,
a contradiction to the dec-minimality of $x$ in $S$.
\end{proof}

\begin{claim}    \label{CLbox2A}
$|W_{2}| = 2 |N' \cap V_{2}|$.
\end{claim}
\begin{proof}
Recall from \eqref{W2def} that $W_{2}$ consists of two parts.
For the first part we have
\begin{equation}   \label{boxW2part1}
 \{ i \in N' \mid x_{i} = \alpha_{2}, \  y_{i} = \alpha_{2}-1 \} 
 =\{ i \in N' \mid x_{i} = \alpha_{2} \} = N' \cap V_{2}
\end{equation}
since
$p \in N' \setminus V_{2}$
by Claim~\ref{CLbox2}.
For the second part, we have 
\begin{align}   
& \{ i \in N' \mid x_{i} = \alpha_{2}-1,  \  y_{i} = \alpha_{2}\} 
=  \{ i \in N' \mid y_{i} = \alpha_{2} \} ,
\label{boxW2part2a}
\\ &
 | \{ i \in N' \mid y_{i} = \alpha_{2} \} | 
= |N' \cap V_{2}| 
\label{boxW2part2b},
\end{align}
since
$\{ i \in N' \mid x_{i} \le \alpha_{2}-2,  \  y_{i} = \alpha_{2}\} = \emptyset$
from the definition \eqref{pargmax} of $p$ and 
the assumption (A1).
Hence follows 
$|W_{2}| = 2 |N' \cap V_{2}|$.
\end{proof}

Note that 
$N' \cap V_{2} \subseteq W_{2}$ and $p \notin W_{2}$
follow from Claims \ref{CLbox2} and \ref{CLbox2A}.

\begin{claim}    \label{CLbox2B}
For vectors $z\sp{k}$ in \eqref{zconvcombi}, we have
\begin{equation} \label{TmemoEq7B}
| \{ i \in W_{2}  \mid z\sp{k}_{i} = \alpha_{2} \}| 
=
| \{ i \in N'
 \mid z\sp{k}_{i} = \alpha_{2} \}| \ge | N' \cap V_{2}|
\qquad (k=1,2,\ldots,K).
\end{equation}
\end{claim}
\begin{proof}
The equality follows from \eqref{zkalp2} and \eqref{zkalp21}. 
Since
$z\sp{k}_{i} \leq  \alpha_{2}$
for all $i \in N'$ and
$(x|_{W_{1}}){\downarrow} = (z\sp{k}|_{W_{1}}){\downarrow}$
by Claim~\ref{CLbox1C},
the inequality
$| \{ i \in N' \mid z\sp{k}_{i} = \alpha_{2} \}| < | N' \cap V_{2}|$
would imply 
 $z\sp{k} <_{\rm dec} x$,
contradicting the dec-minimality of $x$.
\end{proof}

\begin{claim}    \label{CLbox2C}
$(x|_{W_{2}}){\downarrow} = (y |_{W_{2}}){\downarrow}= 
(z\sp{k}|_{W_{2}}){\downarrow}$
\quad $(k=1,2,\ldots,K)$. 
\end{claim}

\begin{proof}
Since
$z_{i} = \alpha_{2} - 1/2$ for all $i \in W_{2}$
and 
$|W_{2}| = 2 |N' \cap V_{2}|$
by Claim \ref{CLbox2A},
we have
\begin{align*}
 z(W_{2}) &= |W_{2}| (\alpha_{2} - 1/2) 
= (2 \alpha_{2} - 1) |N' \cap V_{2}|.
\end{align*}
On the other hand, it follows from
\eqref{zconvcombi} and
Claims \ref{CLbox2A} and \ref{CLbox2B}
that
\begin{align*}
 z(W_{2}) &= 
\sum_{k=1}\sp{K} \lambda_{k} z\sp{k}(W_{2})
\ge \sum_{k=1}\sp{K} \lambda_{k} 
(\alpha_{2} |N' \cap V_{2}| + (\alpha_{2}-1) |N' \cap V_{2}|)
= (2 \alpha_{2} -1 )|N' \cap V_{2}| .
\end{align*}
Therefore,
equality holds in \eqref{TmemoEq7B},
that is,
\begin{equation*}  
 | \{ i \in W_{2} \mid z\sp{k}_{i} = \alpha_{2} \}|
= | \{ i \in W_{2} \mid z\sp{k}_{i} = \alpha_{2} -1 \}|
= | N' \cap V_{2}|
\qquad (k=1,2,\ldots,K).
\end{equation*} 
Combining this with
\eqref{boxW2part1},
\eqref{boxW2part2a},  and \eqref{boxW2part2b},
we obtain
$(x|_{W_{2}}){\downarrow} = (y |_{W_{2}}){\downarrow}= 
(z\sp{k}|_{W_{2}}){\downarrow}$.
\end{proof}

In Steps 1 and 2, 
we have shown $p \notin V_{1}$
(Claim~\ref{CLbox1})
and $p \notin W_{1} \cup V_{2}$ (Claim~\ref{CLbox2}).
By Claims \ref{CLbox1C} and \ref{CLbox2C},
we may go on to Step 3 concentrating on 
the components 
of the vectors
$x$, $y$, and $z\sp{k}$ \ $(k=1,2,\ldots,K)$
within $N' \setminus W_{2} = N \setminus (W_{1} \cup W_{2})$,
where we can show
$p \notin W_{1} \cup W_{2} \cup V_{3}$.
Continuing this way until Step $r$,
we obtain
$p \notin (W_{1} \cup W_{2}  \cup \cdots \cup W_{r-1}) \cup V_{r} = N$,
which is a contradiction.
This completes the proof of
Proposition~\ref{PRdecminICunit}.

\subsection{Determining a face}
\label{SCproofAltFace}

We show how to construct the face $F$ in Theorem~\ref{THdecminICunit}.
We rely on the convex characterization 
of dec-min elements in Theorem~\ref{THdecminRapid}
and Farkas' lemma,
while avoiding using the Fenchel-type duality (Theorem~\ref{THfencICsetsep}).

We first state a variant of Farkas' lemma.

\begin{lemma}     \label{LMfarkasvarIC}
For any matrix $C$ and vector $d$,
the following conditions 
{\rm (a)} and {\rm (b)} are equivalent:
\begin{enumerate}
\item[{\rm (a)}]
There exists a vector $q$ that satisfies $C q \ge d$.

\item[{\rm (b)}]
There exists no vector $r$ that satisfies 
\begin{equation} \label{farkasD}
r\sp{\top} C = \veczero\sp{\top}, 
\quad
r \ge \veczero, 
\quad
r\sp{\top} \vecone = 1,
\quad
r\sp{\top} d > 0.
\end{equation}
\end{enumerate}
\end{lemma}
\begin{proof}
A variant of Farkas' lemma given in \cite[Corollary 7.1e]{Sch86} reads:
Let $A$ be a matrix and let $b$ be a vector.  
Then the system $A x \le b$ of linear inequalities has a solution $x$,
if and only if $y\sp{\top} b \ge 0$ 
for each vector $y \ge \veczero$ with $y\sp{\top} A=\veczero$.
By replacing $(A,b)$ to $(-C,-d)$ and
$(x,y)$ to $(q,r)$
and normalizing $r$ by 
$r\sp{\top} \vecone = 1$,
we obtain the statement of the lemma.
\end{proof}

\medskip

Let $B\sp{\circ} = [a, b]_{\ZZ}$ 
denote the smallest integral box containing $\decmin(S)$,
where
$a \in \ZZ\sp{n}$
and $b \in \ZZ\sp{n}$ denote
the minimum and maximum elements of $B\sp{\circ}$,
respectively.
We have 
$\| a - b \|_{\infty} \le 1$
from Proposition~\ref{PRdecminICunit}.
Using a rapidly increasing function $\varphi: \ZZ \to \RR$ 
in Theorem~\ref{THdecminRapid},
define a vector $p \in \RR\sp{n}$ by
\begin{equation} \label{diffphiGD}
p_{i} = \varphi(a_{i}+1) - \varphi(a_{i})
\qquad (i=1,2,\ldots,n),
\end{equation}
for which 
$\argmin (\varphi[-p_{i}]) = \{ a_{i}, a_{i} + 1 \}$
for $i=1,2,\ldots,n$.
Then we have
\begin{equation} \label{faceAB}
\decmin(S) \subseteq S \cap B\sp{\circ}  \subseteq S \cap \argmin (\Phi_{\rm rap}[-p]).
\end{equation}

\begin{claim}    \label{CLface2GD}
$\displaystyle \decmin(S) 
= \argmin \{  \langle p, y \rangle  \mid  y \in S \cap B\sp{\circ} \} $.
\end{claim}
\begin{proof}
Let $x \in \decmin(S)$ and $y \in S \cap B\sp{\circ}$.
By \eqref{faceAB} we have
$\Phi_{\rm rap}[-p](x) = \Phi_{\rm rap}[-p](y)$,
that is,
\[
\Phi_{\rm rap}(x) - \langle p, x \rangle 
= \Phi_{\rm rap}(y) - \langle p, y \rangle .
\]
Since 
$\Phi_{\rm rap}(x) \le \Phi_{\rm rap}(y)$
by $x \in \decmin(S)$, we have
$\langle p, x \rangle \le \langle p, y \rangle$.
Moreover, we have
\[
\langle p, x \rangle = \langle p, y \rangle
\iff
\Phi_{\rm rap}(x) = \Phi_{\rm rap}(y)
\iff
y \in \decmin(S).
\]
Hence follows
$\displaystyle \decmin(S) =  
 \argmin \{ \langle p, y \rangle  \mid y \in S \cap B\sp{\circ} \}$.
\end{proof}

\begin{figure} 
\centering
 \includegraphics[width=0.70\textwidth,clip]{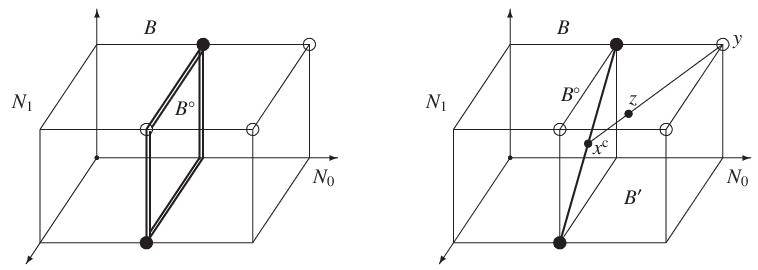}
 \caption{Definition of boxes $B\sp{\circ}$, $B$, and $B'$;
$\bullet \in \decmin(S)$, \  $\circ \in S \setminus \decmin(S)$}
 \label{FGdecminBox1}
\end{figure}

We use the following notations (see Fig.~\ref{FGdecminBox1}): 
\begin{align}
& N_{0} = \{ i \in N \mid b_{i} = a_{i} \},
\qquad
 N_{1} = \{ i \in N \mid b_{i} = a_{i} + 1 \},
\label{boxN0N1defGD}
\\ &
B  = \{ z \in \ZZ\sp{n}  \mid 
 a_{i} -1 \le z_{i} \le a_{i} + 1 \  (i \in  N_{0}), \ 
 a_{i} \le z_{i} \le b_{i} \  (i \in  N_{1})   \},
\label{boxBdefGD}
\\ &
  (S \cap B) \setminus \decmin(S) = \{ y\sp{k} \mid k=1,2,\ldots, L \},
\label{ykdefGD}
\end{align}
where \eqref{ykdefGD} means that we denote the elements of 
$(S \cap B) \setminus \decmin(S)$ by 
$y\sp{k}$ ($k=1,2,\ldots, L)$.

Fix $x\sp{\circ} \in \decmin(S)$.
We are going to modify 
$p$ to $p\sp{*}$ 
so that
\begin{equation} \label{pperturbIneq}
 \langle p\sp{*}, x\sp{\circ} \rangle \le  \langle p\sp{*}, y\sp{k} \rangle
\qquad (k=1,2,\ldots, L)
\end{equation}
holds.
We assume that the components of $p$ on $N_{0}$ 
are changed with an appropriate $q \in \RR\sp{N_{0}}$ as 
\begin{equation} \label{pperturb}
p\sp{*}_{i} =
\begin{cases}
  p_{i} + q_{i}  & (i \in N_{0}),
\\
 p_{i}  & (i \in N_{1}).
\end{cases}
\end{equation}
This definition can also be expressed as
$p\sp{*} = (p|_{N_{0}} + q, p|_{N_{1}})$,
where
$p|_{N_{0}}$ and $p|_{N_{1}}$
represent the restrictions of $p$ to $N_{0}$ and $N_{1}$,
respectively.

\begin{claim}     \label{PRface5GD}
There exists $q$ for which $p\sp{*}$ in \eqref{pperturb} satisfies \eqref{pperturbIneq}.
\end{claim}

\begin{proof}
For each $k=1,2,\ldots, L$, we have
\begin{align} 
 \langle p\sp{*}, x\sp{\circ} \rangle \le  \langle p\sp{*}, y\sp{k} \rangle
& \iff
 \langle p, y\sp{k} \rangle - 
 \langle p, x\sp{\circ} \rangle 
+
 \langle q, y\sp{k}|_{N_{0}} \rangle - 
 \langle q, x\sp{\circ}|_{N_{0}} \rangle 
 \ge 0
\nonumber \\ & \iff
 \langle q, (y\sp{k}  - x\sp{\circ})|_{N_{0}} \rangle  \ge
 \langle p, x\sp{\circ} - y\sp{k} \rangle
\nonumber \\ & \iff
 \langle q, c\sp{k} \rangle \ge d_{k},
\label{pperturbIneq2}
\end{align}
where
\begin{equation} \label{pperturbIneq2ckdk}
c\sp{k}  = (y\sp{k} - x\sp{\circ})|_{N_{0}},
\qquad
d_{k} = \langle p, x\sp{\circ} - y\sp{k} \rangle .
\end{equation}
Let $C$ be an $L \times |N_{0}|$ matrix
whose $k$th row is given by
$(c\sp{k})\sp{\top} \in \{ -1,0,+1 \}\sp{N_{0}}$
for $k=1,2,\ldots, L$,
and let
$d = (d_{1},d_{2}, \ldots, d_{L}) \in \RR\sp{L}$. 
Then \eqref{pperturbIneq2} is expressed as
\begin{equation} \label{pperturbIneq2mat}
 C q \ge d .
\end{equation}
By Lemma~\ref{LMfarkasvarIC}
(a variant of Farkas' lemma),
the inequality system
$C q \ge d$ has a solution $q$
if and only if there exists no 
$r \in \RR\sp{L}$ satisfying \eqref{farkasD}.

Let 
$r = (r_{1},r_{2}, \ldots, r_{L})$
be any vector
that satisfies  
the conditions
$r\sp{\top} C = \veczero\sp{\top}$, 
$r \ge \veczero$,
$r\sp{\top} \vecone = 1$
in \eqref{farkasD},
excepting the inequality condition $r\sp{\top} d > 0$.
Define
\begin{equation} \label{farkasDz}
z = \sum_{k=1}\sp{L} r_{k} y\sp{k}.
\end{equation}
For $i \in N_{0}$, we have
\[
 z_{i} = \sum_{k=1}\sp{L} r_{k} y\sp{k}_{i}
 =  \sum_{k=1}\sp{L} r_{k} x\sp{\circ}_{i} = x\sp{\circ}_{i} =  a_{i}
\]
since
$r\sp{\top} C = \veczero\sp{\top}$
and $x\sp{\circ} \in B\sp{\circ}$.
For $i \in N_{1}$, we have
\[
a_{i} \le z_{i}  \le b_{i}
\]
from 
the definition \eqref{ykdefGD} of $B$,
$y\sp{k} \in B$,
and
$a_{i} \le y\sp{k}_{i} \le b_{i}$.
Hence, $z \in \overline{B\sp{\circ}}$.
In addition, 
we have $z \in \overline{S}$
since
$y\sp{k} \in S$ ($k=1,2,\ldots, L$).
A combination of these two 
as well as the integral convexity of $S$
implies that
$z \in \overline{S} \cap \overline{B\sp{\circ}}= \overline{S \cap B\sp{\circ}}$.

We now turn to the remaining inequality condition in \eqref{farkasD}.
Note that
\[
 r\sp{\top} d 
= \sum_{k=1}\sp{L} r_{k} 
 \langle p, x\sp{\circ} - y\sp{k} \rangle
= \sum_{k=1}\sp{L} r_{k} 
( \langle p, x\sp{\circ} \rangle
 - \langle p, y\sp{k} \rangle )
= \langle p, x\sp{\circ} \rangle
 - \langle p, z \rangle .
\]
Here we have
$\langle p, x\sp{\circ} \rangle
 - \langle p, z \rangle \le 0$
by Claim~\ref{CLface2GD},
since
$x\sp{\circ} \in \decmin(S)$ and
$z \in \overline{S \cap B\sp{\circ}}$.
Thus we have shown that there exists no $r$ satisfying all conditions 
in \eqref{farkasD}.
This implies, by Lemma~\ref{LMfarkasvarIC}, that
there exists $q$ satisfying $C q \ge d$,
which is equivalent to saying that 
there exists $q$ satisfying \eqref{pperturbIneq}.
\end{proof}

\begin{claim}     \label{CLface2GDstar}
$\displaystyle \decmin(S) 
=  \argmin \{  \langle p\sp{*}, y \rangle \mid  y \in S \cap B\sp{\circ} \}$.
\end{claim}
\begin{proof}
Let $y$ be any vector in $S \cap B\sp{\circ}$.
Since $y|_{N_{0}} = a|_{N_{0}}$
and $p\sp{*}$ is given in the form of \eqref{pperturb}, 
we have
\[
\langle p\sp{*}, y \rangle
= \langle p, y \rangle + \langle q, y|_{N_{0}}  \rangle 
= \langle p, y \rangle + \langle q, a|_{N_{0}} \rangle ,
\]
from which
\[
\argmin \{  \langle p\sp{*}, y \rangle  \mid y \in S \cap B\sp{\circ} \} 
=\argmin \{  \langle p, y \rangle  \mid y \in S \cap B\sp{\circ} \} .
\]
By Claim~\ref{CLface2GD},
the right-hand side coincides with $\decmin(S)$. 
\end{proof}

\begin{claim}     \label{CLface4GD}
$\displaystyle \decmin(S)  \subseteq 
\argmin \{  \langle p\sp{*}, y \rangle \mid  y \in S \cap B \}$.
\end{claim}
\begin{proof}
For any $y \in S \cap B$,
we have
$\langle p\sp{*}, x\sp{\circ} \rangle \le \langle p\sp{*}, y \rangle $
by \eqref{pperturbIneq},
whereas
Claim~\ref{CLface2GDstar} shows
$\displaystyle 
\langle p\sp{*}, x\sp{\circ} \rangle = 
\min \{  \langle p\sp{*}, y \rangle  \mid y \in S \cap B\sp{\circ} \}$.
By
$S \cap B \supseteq S \cap B\sp{\circ}$
and Claim~\ref{CLface2GDstar},
we obtain
\[
\argmin \{ \langle p\sp{*}, y \rangle \mid y \in S \cap B \}  
\supseteq
 \argmin \{ \langle p\sp{*}, y \rangle \mid y \in S \cap B\sp{\circ} \} 
= \decmin(S).
\]
\end{proof}

\begin{claim}     \label{CLface3GD}
$\displaystyle \decmin(S) \subseteq 
 \argmin \{ \langle p\sp{*}, y \rangle \mid y \in S \}$.
\end{claim}
\begin{proof}
Let
$\beta =  \min \{  \langle p\sp{*}, x \rangle  \mid x \in \decmin(S) \}$
and let $x\sp{\rm c}$ be the barycenter of $\decmin(S)$, which is defined by
\[
x\sp{\rm c} = \frac{1}{|\decmin(S)|} \sum_{x \in \decmin(S)} x
\]
(see Fig.~\ref{FGdecminBox1}(right)).
We have
$x\sp{\rm c} \in \overline{S} \cap \overline{B\sp{\circ}}= \overline{S \cap B\sp{\circ}}$
by the integral convexity of $S$.

Take any $y \in S$.
For a sufficiently small $\varepsilon > 0$, define
$z = (1 - \varepsilon) x\sp{\rm c} + \varepsilon y$.
Since
$x\sp{\rm c}$ is an interior point of $\overline{B}$
and  
$\varepsilon > 0$ is sufficiently small,
$z$ is also an interior point of $\overline{B}$.
Hence  
$z$ is contained in the convex hull of 
$B' = [ \lfloor z \rfloor,  \lfloor z \rfloor + \vecone ]_{\ZZ}$,
which is a unit box contained in $B$.
This implies that
${\rm N}(z) \subseteq B'$
and hence
\[
 z \in \overline{S \cap {\rm N}(z)} \subseteq  \overline{S \cap B'}
\]
by the integral convexity of $S$.
We can represent $z$ as a convex combination of points in $S \cap B'$ as
\[
z = \sum_{k=1}\sp{K} \lambda_{k} x\sp{k},
\quad
x\sp{k} \in S \cap B',
\quad \sum_{k=1}\sp{K} \lambda_{k}  = 1,
\quad \lambda_{k} > 0 \ \ (k=1,2,\ldots, K).
\]
Since
$\min \{  \langle p\sp{*}, y \rangle  \mid y \in S \cap B' \} \ge \beta$
by Claim \ref{CLface4GD},
we have
$\langle p\sp{*}, x\sp{k} \rangle \ge \beta$
for all $k$.
Thus we obtain
$\langle p\sp{*}, z \rangle \ge \beta$,
while
$\langle p\sp{*}, z \rangle 
=\langle p\sp{*}, (1 - \varepsilon) x\sp{\rm c} + \varepsilon y\rangle
= (1 - \varepsilon) \beta + \varepsilon \langle p\sp{*}, y \rangle $.
Therefore
$\langle p\sp{*}, y \rangle \ge \beta$,
from which the claim follows.
\end{proof}

Let 
\begin{equation} \label{faceFdef}
 F= \argmin \{   \langle p\sp{*}, y \rangle  \mid y \in \overline{S} \}
 = \{ y \in \overline{S} \mid  \langle p\sp{*}, y \rangle = \beta  \},
\end{equation}
which is a face of $\overline{S}$.
It follows from Claims \ref{CLface2GDstar} and \ref{CLface3GD} that
\begin{align}
\decmin(S) 
& = 
\argmin_{ y \in S \cap B\sp{\circ}} ( \langle p\sp{*}, y \rangle )
= \argmin_{y \in S} ( \langle p\sp{*}, y \rangle ) \cap B\sp{\circ}
= F \cap B\sp{\circ}.
 \label{fnsumargminIC4GD}
\end{align}

As a summary of the above argument we obtain the following.
\begin{proposition}    \label{PRface1GD}
Let $S$ be an integrally convex set admitting a dec-min element.
Then $\decmin(S) = F \cap B\sp{\circ}$,
where
$F$ is a face of $\overline{S}$ given by \eqref{faceFdef}
and 
$B\sp{\circ}$ is the smallest integral box containing $\decmin(S)$.
\end{proposition}

By Proposition~\ref{PRdecminICunit},
$B\sp{\circ}$ is a unit box (having the $L_{\infty}$-diameter bounded by 1).
Thus we have completed an alternative proof of Theorem~\ref{THdecminICunit}
by elementary tools.

\begin{example} \rm \label{EXdecminIC4faceprf}
We illustrate the above argument for a simple example.

\begin{enumerate}
\item
Consider an integrally convex set
$S= \{ 
(2,0,0,0), 
(1,1,0,1), 
(1,0,1,1), 
(0,1,1,2)
 \}$.
The four points lie on a two-dimensional plane
in the four-dimensional space,
and they are actually the vertices of a parallelogram. 
We have
\[
\decmin(S) 
= \{ (1,1,0,1), (1,0,1,1)  \}= \{ x\sp{1}, x\sp{2} \},
\]
where $x\sp{1} = (1,1,0,1)$, $x\sp{2} =(1,0,1,1)$.
The minimal cube $B\sp{\circ}$ containing $\decmin(S)$
is given by
\[
B\sp{\circ} 
= [a, b]_{\ZZ} 
= [(1,0,0,1), (1,1,1,1)]_{\ZZ} 
= \{ 1 \} \times \{ 0,1 \}\sp{2} \times \{ 1 \}
\]
with $a = (1,0,0,1)$ and $b=(1,1,1,1)$.
For 
\eqref{boxN0N1defGD},
\eqref{boxBdefGD},
\eqref{ykdefGD}, we have
\begin{align*}
& N_{0} = \{ 1, 4 \},  \quad
N_{1} = \{ 2, 3 \},
\\
& B = [(0,0,0,0), (2,1,1,2)]_{\ZZ} 
= \{ 0,1,2 \} \times \{ 0,1 \}\sp{2} \times \{ 0,1,2 \},
\\ &
(S \cap B) \setminus \decmin(S)  
 = \{ (2,0,0,0), (0,1,1,2) \}  = \{ y\sp{1}, y\sp{2} \},
\end{align*}
where $ y\sp{1}= (2,0,0,0)$, $y\sp{2} = (0,1,1,2)$.

\item
By choosing $\varphi(k) = 10\sp{k}$, we have
\begin{align*}
& \Phi_{\rm rap}((2,0,0,0)) = \varphi(2) + 3 \varphi(0) = 103,
\\ &
\Phi_{\rm rap}((1,1,0,1))= \Phi_{\rm rap}((1,0,1,1))
 = 3 \varphi(1) + \varphi(0) = 31,
\\ &
\Phi_{\rm rap}((0,1,1,2)) = \varphi(2) + 2 \varphi(1) + \varphi(0) = 121. 
\end{align*}
We have $\decmin(S) =\argmin (\Phi_{\rm rap} | S)$
as in Theorem~\ref{THdecminRapid}(3).

\item
For $p_{i} = \varphi(a_{i}+1) - \varphi(a_{i})$
in \eqref{diffphiGD}, we have
\[
p_{1} = p_{4} = \varphi(2) - \varphi(1) = 90,
\qquad
p_{2} = p_{3} = \varphi(1) - \varphi(0) = 9,
\]
that is, $p =(90, 9, 9, 90)$. For this $p$ we have
\begin{align*}
& \Phi_{\rm rap}[-p]((2,0,0,0)) = 103 -2 p_{1} = 103- 180 = -77,
\\ &
\Phi_{\rm rap}[-p]((1,1,0,1))
 = 31 - (p_{1}+p_{2}+p_{4}) = 31 - 189 = -158,
\\ &
\Phi_{\rm rap}[-p]((1,0,1,1))
 = 31 - (p_{1}+p_{3}+p_{4}) = 31 - 189 = -158,
\\ &
\Phi_{\rm rap}[-p]((0,1,1,2)) 
 = 121 - (p_{2}+p_{3}+2 p_{4}) = 121 -  198 = -77,
\end{align*}
from which
\[
S \cap \argmin (\Phi_{\rm rap}[-p]) = \{ (1,1,0,1), (1,0,1,1)  \} .
\]
We thus have equality in the inclusion relation
$S \cap B\sp{\circ} \subseteq S \cap \argmin (\Phi_{\rm rap}[-p])$
in \eqref{faceAB}.

\item
The inner product
$\langle p, x \rangle$
takes the following values 
for $x \in S$:
\[
\langle p, x\sp{1} \rangle = \langle p, x\sp{2} \rangle = 189,
\quad 
\langle p, y\sp{1} \rangle = 180,
\quad 
\langle p, y\sp{2} \rangle = 198.
\]
The elements of $\decmin(S) = \{ x\sp{1}, x\sp{2} \}$
do not minimize $\langle p, x \rangle$ over $S$.
In the following we modify $p$ to $p\sp{*}$ 
so that
the elements of $\decmin(S)$
minimize $\langle p\sp{*}, x \rangle$.

\item
We consider $C q \ge d$ in \eqref{pperturbIneq2mat}
with the choice of $x\sp{\circ} = (1,1,0,1) \ (=x\sp{1})$.
According to \eqref{pperturbIneq2ckdk} we have
\begin{align*}
c\sp{1}  &= (y\sp{1} - x\sp{\circ})|_{N_{0}} = (2,0)-(1,1) = (1,-1), 
\\
c\sp{2}  &= (y\sp{2} - x\sp{\circ})|_{N_{0}} = (0,2)-(1,1) = (-1,1), 
\\
d_{1} &= \langle p, x\sp{\circ} - y\sp{1} \rangle 
 = \langle p, (-1,1,0,1) \rangle = 9,
\\
d_{2} &= \langle p, x\sp{\circ} - y\sp{2} \rangle 
 = \langle p, (1,0,-1,-1) \rangle = -9,
\end{align*}
with which
\[
C q \ge d 
\ \Leftrightarrow \ 
\left[
\begin{array}{rr}
     1 & -1  \\
    -1 &  1 \\
\end{array} 
\right] 
\left[
\begin{array}{c}
     q_{1}  \\
     q_{2} \\
\end{array} 
\right]
\ge
\left[
\begin{array}{r}
     9  \\
    -9 \\
\end{array} 
\right] 
\ \Leftrightarrow \ 
q_{1} -  q_{2} = 9
\ \Leftrightarrow \ 
(q_{1}, q_{2}) = (\alpha + 9, \alpha).
\]
That is, $p\sp{*} = (99 + \alpha, 9, 9, 90 + \alpha)$
with any $\alpha \in \RR$.

\item
On noting
\[
 \langle p\sp{*}, x\sp{1} \rangle 
= \langle p\sp{*}, x\sp{2} \rangle 
= \langle p\sp{*}, y\sp{1} \rangle 
= \langle p\sp{*}, y\sp{2} \rangle = 198 + 2 \alpha 
\]
and recalling 
$\decmin(S) = \{ x\sp{1}, x\sp{2} \}$, we can verify the claims:

Claim~\ref{CLface2GDstar}:
$\displaystyle \decmin(S) 
=  \argmin_{y \in S \cap B\sp{\circ} }( \langle p\sp{*}, y \rangle )$
$= S \cap B\sp{\circ} = \{ x\sp{1}, x\sp{2} \}$.

Claim~\ref{CLface4GD}:
$\displaystyle \decmin(S) 
 \subseteq \argmin_{y \in S \cap B} ( \langle p\sp{*}, y \rangle )$
$= S \cap B = \{ x\sp{1}, x\sp{2}, y\sp{1}, y\sp{2} \}$.

Claim~\ref{CLface3GD}:
$\displaystyle \decmin(S) \subseteq 
 \argmin_{y \in S} ( \langle p\sp{*}, y \rangle )$
$= S = \{ x\sp{1}, x\sp{2}, y\sp{1}, y\sp{2} \}$.

\item
The face $F$
in Proposition~\ref{PRface1GD}
is given by
$\displaystyle F= \argmin \{  \langle p\sp{*}, y \rangle  \mid y \in \overline{S} \}$.
In this particular example, we have
$F = \overline{S}$.
We thus obtain the desired representation 
$\decmin(S) = F \cap B\sp{\circ}$.
\finbox
\end{enumerate}
\end{example}


\bigskip


\noindent {\bf Acknowledgement}.
This work was supported by JSPS/MEXT KAKENHI JP23K11001 and JP21H04979, and
by JST ERATO Grant Number JPMJER2301, Japan.





\end{document}